\documentclass[a4paper,11pt]{article}
	\usepackage[utf8]{inputenc}
	\usepackage{hyperref}
	\usepackage{amssymb}
	\usepackage{amsmath}
	\usepackage{amsfonts}
	\usepackage{amsthm}
	\usepackage{stmaryrd}
	\usepackage{algorithm,algorithmic,refcount}
	\usepackage{xcolor}
	\usepackage{tkz-euclide}

	\usepackage{epstopdf}
	\usepackage{mathdots}
	\usepackage{appendix}
	
	\usepackage{tikz}
	\usetikzlibrary{matrix,decorations.pathreplacing,calc}

	\usepackage{natbib}
	\setcitestyle{square}
	\setcitestyle{numbers}
	\setcitestyle{comma}
	
	\newcommand{\R}{\mathbb {R}}
	\newcommand{\C}{\mathbb {C}}
	\newtheorem{theorem}{Theorem}
	\newtheorem{example}[theorem]{Example}
	
	\newtheorem{lemma}[theorem]{Lemma}
	\newtheorem{remark}[theorem]{Remark}

    \newcommand{\rev}[1]{\textcolor{blue}{#1}}

\usepackage[normalem]{ulem}

    \usepackage{amssymb,amsmath,url}
    \author{Alice Cortinovis\footnote{Department of Mathematics, Stanford University, Stanford, CA 94305, USA. The work of Alice Cortinovis has been partially supported by the SNSF research project \emph{Fast algorithms from low-rank updates}, grant number: 200020\_178806. E-mail: alicecor@stanford.edu. } \and Daniel Kressner\footnote{Institute of Mathematics, EPFL, CH-1015 Lausanne, Switzerland. E-mail: daniel.kressner@epfl.ch} \and Yuji Nakatsukasa\footnote{Mathematical Institute, University of Oxford, Oxford, OX2 6GG, UK. E-mail: nakatsukasa@maths.ox.ac.uk}}
    \date{}
	  
	\title{Speeding up Krylov subspace methods \\ for computing $f(A)b$  via randomization}
	
\begin{document}
	
	\maketitle

\begin{abstract}
This work is concerned with the computation of the action of a matrix function $f(A)$, such as the matrix exponential or the matrix square root, on a vector $b$. For a general matrix $A$, this can be done by computing the compression of $A$ onto a suitable Krylov subspace. Such compression is usually computed by forming an orthonormal basis of the Krylov subspace using the Arnoldi method. In this work, we propose to compute (non-orthonormal) bases in a faster way and to use a fast randomized algorithm for least-squares problems to compute the compression of $A$ onto the Krylov subspace. We present some numerical examples which show that our algorithms can be faster than the standard Arnoldi method while achieving comparable accuracy.
\end{abstract}

\section{Introduction}

Given a matrix $A \in \R^{n \times n}$ and a function $f$ which is analytic on the closure of a domain $\mathbb{D} \subset \C$ that contains the eigenvalues of $A$, the matrix function $f(A)$ is defined as
\begin{equation*}%\label{eq:contourdefinition}
	f(A) := \frac{1}{2\pi i} \int_{\partial \mathbb{D}} f(z) (zI - A)^{-1} \mathrm{d} z,
\end{equation*}
where $I$ denotes the identity matrix; see the book~\cite{Higham2008} for equivalent definitions of matrix functions. 

Matrix functions arise, for instance, in the solution of partial differential equations~\cite{Botchev2020,Ciegis2018}, in data analysis~\cite{Benzi2020, Estrada2010}, and in electronic structure calculations~\cite{Benzi2013}. In many of these applications, one is only interested in the product of $f(A)$ with a vector $b \in \mathbb{R}^n$. If $A$ is large, computing the full $f(A)$ by standard methods (e.g. Schur-Parlett algorithm~\cite{Davies2003}) is infeasible in practice. Moreover, the matrix $A$ may itself be too large to store explicitly. When $A$ is only available through matrix-vector multiplications, a Krylov subspace method is usually the method of choice to approximate $f(A)b$ (see, e.g.,~\cite[Chapter 13.2]{Higham2008}); we review this approach in Section~\ref{sec:functions}.

A key ingredient of Krylov subspace methods is the \rev{construction} of an orthonormal basis of the Krylov subspace
\begin{equation} \label{eq:krylovspace}
    \mathcal{K}_m(A,b) = \mathrm{span}\{b, Ab, A^2b, \ldots, A^{m-1}b\}
\end{equation}
for some $m \in \mathbb{N}$ ($m \ll n$), using the Arnoldi algorithm (see, e.g.,~\cite[Section 6.2]{Saad1992}). However, the cost of orthonormalization grows quadratically with $m$, the number of iterations. 
Many approaches have been proposed to deal with this problem, including explicit polynomial approximation, restarting procedures, and nonsymmetric Lanczos methods; see~\cite{Guettel2020} for a recent survey on this topic.

In this paper, we exploit randomization to speed up Krylov subspace methods for computing $f(A)b$. This extends the case of linear systems and eigenvalue problems, which was considered by Balabanov and Grigori in~\cite{Balabanov2020} and by Nakatsukasa and Tropp in~\cite{Nakatsukasa2021}. The recent work by G{\"u}ttel and Schweitzer~\cite{Guettel2022} also addresses the computation of $f(A)b$; we will comment in more detail in Section~\ref{sec:guettelcompare} on its relation to our work.

The key idea of our approach is that there is no need to maintain an orthonormal basis of~\eqref{eq:krylovspace} in order to use Krylov subspace methods, as long as one can efficiently compute the compression of $A$ onto the Krylov subspace. On the one hand, we can exploit faster ways to construct a basis of $\mathcal{K}_m(A,b)$ (as discussed in~\cite{Balabanov2020,Nakatsukasa2021} and in Section~\ref{sec:krylovbasis}), on the other hand we can leverage randomized techniques (such as Blendenpik~\cite{Blendenpik}) in order to compute the compression quickly. We also propose an algorithm that partly bypasses the computation of the compression and is, in turn, even faster -- but comes with weaker convergence guarantees.

The rest of this work is organized as follows. In Section~\ref{sec:krylovbasis} we review and discuss two algorithms from~\cite{Balabanov2020,Nakatsukasa2021} to construct nonorthonormal bases of $\mathcal{K}_m(A,b)$. In Section~\ref{sec:functions} we recall the standard Krylov subspace method and we present our proposed algorithms. In Section~\ref{sec:convergence} we present some convergence properties of our methods and discuss their numerical stability. Section~\ref{sec:numexp} contains several numerical examples that illustrate the convergence and timings of the proposed algorithms and a comparison with an algorithm from~\cite{Guettel2022}. Conclusions are summarized in Section~\ref{sec:conclusion}. \rev{While our algorithms and results are formulated for real matrices only, most of the discussion carries directly over to complex matrices, under a suitable choice of sketching matrix.}

\section{Krylov subspaces and Krylov bases}\label{sec:krylovbasis}

In this section, we discuss the computation of an orthonormal basis of $\mathcal{K}_m(A,b)$ \rev{defined in~\eqref{eq:krylovspace}}, the $m$th Krylov subspace associated to a matrix $A \in \R^{n \times n}$ and the vector $b \in \R^n$.% defined as in~\eqref{eq:krylovspace}. 
To simplify the discussion, we will assume that $\mathcal{K}_m(A,b)$ is an $m$-dimensional subspace of $\R^n$, which holds whenever $m < n$ and $b$ is not contained in a (nontrivial) invariant subspace of $A$.
An orthonormal basis $U_m = \begin{bmatrix} u_1 & \cdots & u_m\end{bmatrix}$ of $\mathcal{K}_m(A,b)$ can be computed using the Arnoldi method (see, e.g.,~\cite[Section 6.2]{Saad1992}), together with an upper Hessenberg matrix $K_m \in \R^{m \times m}$ and a unit vector $u_{m+1}$ such that
\begin{equation}\label{eq:arnoldi}
	A U_m = U_m K_m + k_{m+1,m}u_{m+1} e_m^T = U_{m+1} \underline{K}_m, \qquad \underline{K}_m = \begin{bmatrix} K_m \\ k_{m+1,m} e_m^T \end{bmatrix} \in \R^{(m+1) \times m},
\end{equation}
where $e_m$ is the $m$th vector of the canonical basis of $\R^m$.
The relation~\eqref{eq:arnoldi} is called the \emph{Arnoldi decomposition}. Note that we have
$
	U_m^T A U_m = K_m,
$
that is, $K_m$ is the compression of $A$ onto the $m$th Krylov subspace. 
The Arnoldi method requires $\mathcal O(nm^2)$ operations, \rev{in addition} to the $m$ matrix-vector products with $A$. 

\subsection{Non-orthonormal basis of the Krylov subspace}

As $m$ increases, the quadratic dependence on $m$ of the complexity of the Arnoldi method becomes disadvantageous.  In the following, we discuss two other ways of constructing bases of $\mathcal{K}_m(A,b)$, which are not orthonormal, but can be computed faster. In the works~\cite{Balabanov2020,Nakatsukasa2021} such bases were introduced and used in the context of linear systems and eigenvalue problems. In our work, we extend this approach to matrix functions; indeed, as we will see in Section~\ref{sec:functions}, orthonormality is not necessary for approximating $f(A)b$.

The algorithms that we discuss build a basis $V_m$ of $\mathcal{K}_m(A,b)$ by rescaling $b$ and then adding one vector at a time. To enlarge the basis, \rev{the last basis vector is pre-multiplied by $A$}, and then a suitable linear combination of the previous basis vectors is added to the new vector. A rescaled version of the result becomes the new basis vector. Algorithms that fall into the described framework give rise to an Arnoldi-like decomposition
\begin{equation}\label{eq:arnoldilike}
	A V_m = V_m H_m + h_{m+1,m} v_{m+1} e_m^T = V_{m+1} \underline H_m, \qquad \underline H_m = \begin{bmatrix} H_m \\ h_{m+1,m} e_m^T\end{bmatrix} \in \R^{(m+1) \times m},
\end{equation}
where $V_m$ is the computed basis of $\mathcal{K}_m(A,b)$ and $H_m$ is an upper Hessenberg matrix. A key advantage of this more general framework is that it allows for cheaper linear combinations for producing the new basis vector compared to the modified Gram-Schmidt procedure by the Arnoldi method. 
The two methods from~\cite{Balabanov2020,Nakatsukasa2021} belong to this framework and we summarize them in Sections~\ref{sec:BG} and~\ref{sec:NT}, respectively.

\subsubsection{Building the basis by orthogonalizing a sketch}\label{sec:BG}

The idea of the method proposed in~\cite{Balabanov2020} is that instead of orthogonalizing each new basis vector against all the previous basis vectors, one maintains a sketched version of the basis and orthogonalizes a sketch of the new vector against the sketch of the previous basis vectors. We now describe more precisely how the basis is constructed, and the procedure is summarized in Algorithm~\ref{alg:BG}.  First of all we choose a sketch matrix $\Theta \in \R^{s \times n}$, for some $m < s \ll n$.
At the $i$th step we build a new vector $w_i$ by multiplying $A$ with the last vector of the basis. We also keep a sketched version of the basis $S_{i-1} = \Theta V_{i-1}$. We compute the sketch $p_i := \Theta w_i$\rev{,} we perform one step of Gram-Schmidt orthogonalization on $\begin{bmatrix} S_{i-1} & p_i \end{bmatrix}$ (which is a small matrix)\rev{, and we denote by $S_i$ the resulting $s \times i$ matrix (which has orthonormal columns)}. We collect the corresponding coefficients $r_i$ in an upper triangular matrix $R$; see lines~\ref{line:beginGS}--\ref{line:endGS}. Then, we use these same coefficients to approximately orthogonalize the vector $w_i$ with respect to the previous vectors, by computing 
\begin{equation*}
    v_i := (w_i - V_{i-1} r_i)/r_{ii}.
\end{equation*}
The vector $v_i$ is scaled in such a way that its sketched version has unit length, so that the sketched basis $S_i$ remains orthonormal. Under suitable assumptions on the sketch matrix $\Theta$, the basis $V_{m+1}$ obtained by Algorithm~\ref{alg:BG} is well-conditioned; see~\cite[Proposition 4.1]{Balabanov2020}.

We choose $\Theta$ to be a sub-sampled randomized Hadamard transfom (SRHT)~\cite{Tropp2011}: $\Theta$ consists of selecting (uniformly at random) $s$ rows from the product of the Walsh-Hadamard transform with a diagonal matrix $D$ with random i.i.d. Rademacher ($\pm1$) entries on the diagonal. The product of $\Theta$ with a vector can be computed in $\mathcal{O}(n\log n)$ time. We remark that other choices are possible, such as the subsampled randomized Fourier transform or sparse sign matrices; see, e.g.,~\cite[Sections 8--9]{Martinsson2020}. 

Let us compare the cost of Algorithm~\ref{alg:BG} with the cost of the standard Arnoldi algorithm. At the $i$th iteration, both require a matrix-vector product with $A$; Algorithm~\ref{alg:BG} then requires one multiplication with the current basis $V_{i-1}$ (in line~\ref{line:approxorth}), while the standard Arnoldi algorithm requires two multiplications with $V_{i-1}$. All the other computations are comparably inexpensive. Therefore, when, e.g., $A$ is very sparse, we can expect Algorithm~\ref{alg:BG} to be up to twice as fast as  Arnoldi.

\begin{algorithm}
	\caption{Construction of Krylov subspace basis as in~\cite[Algorithm 2]{Balabanov2020}}
	\label{alg:BG}
	\begin{algorithmic}[1]
		\REQUIRE{Matrix $A$, vector $b$, number of iterations $m$, parameter $s \ll n$ (sketch size)}
		\ENSURE{Pair $(V_{m+1}, \underline{H}_m)$ that satisfies Arnoldi-like decomposition~\eqref{eq:arnoldilike}}
		\STATE{Draw sketch matrix $\Theta$ of size $s \times n$}
		\STATE{Initialize $R$ as $(m+1)\times (m+1)$ zero matrix}
		\STATE{$w_1 \leftarrow b$}
		\STATE{$p_1 \leftarrow \Theta w_1$}
		\STATE{$r_{11} \leftarrow \|p_1\|_2$}
		\STATE{$s_1 \leftarrow p_1 / r_{11}$, $S_1 \leftarrow s_1$}
		\STATE{$v_1 \leftarrow w_1/r_{11}$, $V_1 \leftarrow v_1$}
		\FOR {$i = 2, \ldots, m+1$}
		\STATE{$w_i \leftarrow A v_{i-1}$}
		\STATE{$p_i \leftarrow \Theta w_i$}
		\STATE{$r_i \leftarrow S_{i-1}^T p_i$, $R(1:i-1, i)\leftarrow r_i$ }\label{line:beginGS}
		\STATE{$s_i \leftarrow p_i - S_{i-1} r_i$}
		\STATE{$r_{ii} \leftarrow \|s_i\|_2$}
		\STATE{$s_i \leftarrow s_i / r_{ii}$, $S_i \leftarrow \begin{bmatrix} S_{i-1} & s_i \end{bmatrix}$}\label{line:endGS}
		\STATE{$v_i \leftarrow (w_i - V_{i-1} r_i)/r_{ii}$, $V_i \leftarrow \begin{bmatrix} V_{i-1} & v_i \end{bmatrix}$}\label{line:approxorth}
		\ENDFOR
		\STATE{$\underline{H}_m \leftarrow R(:, 2:m+1)$}
	\end{algorithmic}
\end{algorithm}

\subsubsection{Building the basis by truncated orthogonalization}\label{sec:NT}

In the following, we \rev{describe}  the ``truncated orthogonalization'' method proposed in~\rev{\cite[Chapter 6.4.2]{SaadBook} and used in~\cite{Nakatsukasa2021} in the context of sketching for linear systems and eigenvalue problems. The procedure is summarized} in Algorithm~\ref{alg:NT}. Each new vector $w_i = A v_{i-1}$ in line~\ref{line:newvector} is orthogonalized with respect to the last $k$ vectors of the basis only, where $k$ is a fixed small number, e.g., $k=2$; see lines~\ref{line:begink}--\ref{line:endk}. The matrix $H_m$ is built iteratively using the coefficients of such a partial orthogonalization. When $A$ is symmetric, this method coincides for $k=2$ with the Lanczos method and gives -- in exact arithmetic -- an orthonormal basis. When $A$ is nonsymmetric, the produced basis is not orthogonal in general.

The advantage of Algorithm~\ref{alg:NT} is that it avoids the quadratic dependence on $m$; the computational cost is $\mathcal O(nm)$ plus the cost of computing $m$ matrix-vector multiplications \rev{by} $A$. The main disadvantage, however, is that the basis is prone to become increasingly ill-conditioned and it may actually become numerically rank-deficient already after a few iterations. A potential cure for this behavior is \emph{whitening} (see~\cite[Section 2.2]{Nakatsukasa2021}): One monitors the condition number of the basis (or some related quantity) and once ill-conditioning is observed\rev{, an
approximate orthogonalization is performed.} For this purpose we keep a sketched version of the basis (using a sketch matrix $\Theta \in \R^{s \times n}$), and at the end of each cycle we check whether $\Theta V_{i}$ starts to become too ill-conditioned; if this happens, we compute a QR factorization $\Theta V_{i} = Q_{i} R_{i}$ and consider the basis $V_{i} R_{i}^{-1}$ of $\mathcal K_i(A, b)$ instead, which should now be well-conditioned. Accordingly, the new matrix $\underline H_{i-1}$ in the Arnoldi-like relation~\eqref{eq:arnoldilike} becomes $\underline{H}_{i-1} \leftarrow R_{i} \underline H_{i-1} R_{i-1}^{-1}$, where $R_{i-1}$ is the leading $(i-1) \times (i-1)$ submatrix of $R_{i}$. Potentially, we could continue the construction of the basis from the pair $(V_{i}, \underline{H}_{i-1})$ using Algorithm~\ref{alg:NT} with the updated basis, and perform whitening again whenever it is needed, but we observed that in practice whitening is then needed frequently (every few iterations) and the quality of the basis deteriorates quickly. For this reason, after the first time we whiten the basis, we continue the construction of the basis as in Algorithm~\ref{alg:BG}, using the same sketch matrix $\Theta$.

\begin{remark}
Algorithm~\ref{alg:BG} is mathematically equivalent to Algorithm~\ref{alg:NT} followed by whitening after the last step. However, numerically  Algorithm~\ref{alg:BG} is more stable as whitening requires inverting the matrix $R_{i-1}$ which may be ill-conditioned or even numerically singular. 
\end{remark}

\begin{algorithm}
	\caption{Construction of basis as in~\cite{Nakatsukasa2021} (via truncated orthogonalization)}
	\label{alg:NT}
	\begin{algorithmic}[1]
		\REQUIRE{Matrix $A$, vector $b$, number of iterations $m$, truncation parameter $k$}
		\ENSURE{Pair $(V_{m+1}, \underline{H}_m)$ that satisfies Arnoldi-like decomposition~\eqref{eq:arnoldilike}}
		\STATE{$v_1 \leftarrow b / \|b\|_2$}
		\FOR {$i = 2$ \textbf{to} $m+1$}
		\STATE{$ w_i \leftarrow A v_{i-1}$}\label{line:newvector}
		\FOR{$j=\max\{1, i-k\}$ \textbf{to} $i-1$}\label{line:begink}
			\STATE{$h_{j,i-1} \leftarrow v_j^T w_i$}
			\STATE{$w_i \leftarrow w_i - h_{j,i} v_i$}
		\ENDFOR
		\STATE{$h_{i, i-1} \leftarrow \|w\|_2$}
		\STATE{$v_i \leftarrow w_i / \|w_i\|_2$}
		\ENDFOR\label{line:endk}
		\STATE{$V_{m+1} \leftarrow [v_1, \ldots, v_{m+1}]$}
	\end{algorithmic}
\end{algorithm}

\section{Computing $f(A)b$ via Krylov subspace methods}\label{sec:functions}

When $U_m$ is orthonormal, the classical Krylov subspace projection approach to approximate $f(A)b$ consists in taking
\begin{equation}\label{eq:standardfm}
	f_m := \|b\|_2 U_m f(K_m) e_1, \quad K_m = U_m^T A U_m.
\end{equation}
The convergence of $f_m$ to $f(A)b$ is related to polynomial approximation of $f$ on a suitable region of the complex plane containing the eigenvalues of $A$, and will be reviewed at the beginning of Section~\ref{sec:convergence}.

\subsection{Computing $f_m$ from Arnoldi-like decompositions}

The approximation~\eqref{eq:standardfm} can be expressed in terms of \emph{any} basis $V_m$ of the Krylov subspace $\mathcal{K}_m(A,b)$. We include the following lemma for completeness.

\begin{lemma}\label{lemma:basis}
The approximation $f_m$ defined in~\eqref{eq:standardfm} satisfies
\begin{equation}\label{eq:krylovproj}
	f_m = V_m f(V_m^{\dagger} A V_m) V_m^{\dagger} b,
\end{equation}
for an arbitrary basis $V_m$ of $\mathcal{K}_m(A,b)$,
where $\dagger$ denotes the pseudoinverse of a matrix. 
\end{lemma}
\begin{proof}
As $U_m$ and $V_m$ span the same subspace, there is an invertible matrix $R_m \in \R^{m \times m}$  such that $V_m = U_m R_m$. Thus,
\begin{align*}
	V_m f(V_m^{\dagger} A V_m)V_m^{\dagger}b & = U_m R_m f(R_m^{-1} U_m^T A U_m R_m) R_m^{-1} U_m^T b \\
	& = U_m R_m R_m^{-1} f(U_m^T A U_m)R_m R_m^{-1} U_m^T b\\
	& = \|b\|_2 U_m f(K_m) e_1. \qedhere
\end{align*}
\end{proof}
If an Arnoldi-like decomposition~\eqref{eq:arnoldilike} is available, like in Algorithms~\ref{alg:BG} and~\ref{alg:NT}, the expression~\eqref{eq:krylovproj} simplifies. Indeed, we have that
\begin{equation} \label{eq:relationxx}
	V_m^{\dagger} A V_m = H_m + h_{m+1,m}(V_m^{\dagger} v_{m+1})e_m^T.
\end{equation}
Therefore, only the last column of $V_m^{\dagger} A V_m$ differs from the matrix $H_m$ that is already computed.  
Moreover, we have that $V_m^{\dagger} b = \gamma_b e_1$ is a scalar multiple of $e_1$, with $\gamma_b = r_{11}$ for Algorithm~\ref{alg:BG} and $\gamma_b = \|b\|_2$ for Algorithm~\ref{alg:NT}.

 The vector $y := V_m^{\dagger} v_{m+1}$ appearing in~\eqref{eq:relationxx} is the solution of
\begin{equation}\label{eq:leastsquares}
	y = \arg\min_{z \in \mathbb{R}^m} \|V_m z - v_{m+1}\|_2.
\end{equation}
When $n \gg m$, solving this overdetermined least-squares problem exactly 
%\rrr{would constitute the most expensive operation in determining}{
constitutes a significant part of computing the approximation~\eqref{eq:krylovproj}. 

\subsubsection{Solving the least-squares problem}

When the basis $V_m$ is well conditioned (e.g., when it is obtained from Algorithm~\ref{alg:BG}), a few iterations of LSQR~\cite{Paige1982} applied to~\eqref{eq:leastsquares} \rev{are} sufficient to get a good approximation of $V_m^\dagger v_{m+1}$. When $V_m$ is not well conditioned (which is frequently the case when using Algorithm~\ref{alg:NT}), LSQR needs to be combined with a preconditioner. For this purpose, we use Blendenpik~\cite{Blendenpik}, which computes a random sketch of the basis using, once again, a SRHT, and then uses the $R$ factor from a QR decomposition of the sketch as a preconditioner. Note that Algorithm~\ref{alg:BG} already orthonormalizes a sketch of the basis (the matrix $S_m$) and therefore the use of Blendenpik, in particular with the same sketch, would be redundant. 
We remark that it is often not necessary to obtain a very accurate solution of the least-squares problem~\eqref{eq:leastsquares}; see Lemma~\ref{lemma:derivative} below. The proposed methods are summarized in Algorithms~\ref{alg:fAb-LS-BG} and~\ref{alg:fAb-LS-NT}.

\begin{algorithm}
	\caption{Krylov subspace approximation of $f(A)b$ using Algorithm~\ref{alg:BG}}
	\label{alg:fAb-LS-BG}
	\begin{algorithmic}[1]
		\REQUIRE{Matrix $A$, vector $b$, number of iterations $m$}
		\ENSURE{Approximation of $f(A)b$}
		\STATE{Get Arnoldi-like decomposition $(V_{m+1}, \underline H_m)$ and $r_{11}$ from Algorithm~\ref{alg:BG}}
		\STATE{Find $y = \arg \min_{z \in \R^m} \|V_m z - v_{m+1}\|_2$ using LSQR}
		\STATE{Return $r_{11} V_m f(H_m + h_{m+1,m}ye_m^T) e_1$}
	\end{algorithmic}
\end{algorithm}

\begin{algorithm}
	\caption{Krylov subspace approximation of $f(A)b$ using Algorithm~\ref{alg:NT}}
	\label{alg:fAb-LS-NT}
	\begin{algorithmic}[1]
		\REQUIRE{Matrix $A$, vector $b$, number of iterations $m$}
		\ENSURE{Approximation of $f(A)b$}
		\STATE{Get Arnoldi-like decomposition $(V_{m+1}, \underline{H}_m)$ from  Algorithm~\ref{alg:NT}}
		\STATE{Find $y = \arg \min_{z \in \R^m} \|V_m z - v_{m+1}\|_2$ using  Blendenpik~\cite{Blendenpik}}
		\STATE{Return $\|b\|_2 V_m f(H_m + h_{m+1,m}ye_m^T) e_1$}
	\end{algorithmic}
\end{algorithm}

\subsection{A cheaper approximation}
An alternative to the approximation~\eqref{eq:krylovproj} is to evaluate $f$ on the matrix $H_m$ instead of $V_m^{\dagger} A V_m=H_m + h_{m+1,m}(V_m^{\dagger} v_{m+1})e_m^T$, that is, to approximate
\begin{equation}\label{eq:noproj}
	f(A)b \approx \tilde f_m := V_m f(H_m) V_m^{\dagger}b.
\end{equation}
This approximation is cheaper to compute because the solution of the least-squares problem is avoided. Note that~\eqref{eq:noproj} coincides with~\eqref{eq:standardfm} if the basis $V_m$ is orthonormal, but in general $H_m \neq V_m^{\dagger} A V_m$. We recall that these two Hessenberg matrices only differ in their last column, see~\eqref{eq:relationxx}, and the heuristics behind the approach~\eqref{eq:noproj} is that this difference may only have a mild influence on the first column of $f(V_m^{\dagger} A V_m)$; see Theorem~\ref{thm:lastcol} below for additional theoretical insight.  This method is summarized in Algorithm~\ref{alg:fAb-cheap}.
\rev{Note that for $f(z) = 1/z$ (that is, when solving a linear system) Algorithm~\ref{alg:fAb-cheap} becomes equalivant to Saad's Incomplete Orthogonalization Method (IOM) described in~\cite[Alg. 6.6]{Saad}

\begin{algorithm}
	\caption{$f(A)b$ without solving least-squares problems}
	\label{alg:fAb-cheap}
	\begin{algorithmic}[1]
		\REQUIRE{Matrix $A$, vector $b$, number of iterations $m$}
		\ENSURE{Approximation of $f(A)b$}
		\STATE{Get Arnoldi-like decomposition $(V_{m+1}, \underline{H}_m)$ from Algorithm~\ref{alg:NT}}
		\STATE{Return $\|b\|_2 V_m f(H_m) e_1$}
	\end{algorithmic}
\end{algorithm}

\subsection{Comparison with related literature:~\cite{Guettel2022}}\label{sec:guettelcompare}

While preparing this work, 
the recent preprint by
G\"uttel and Schweitzer~\cite{Guettel2022} appeared, which proposes -- independently from our work -- two other randomized methods, called sFOM and sGMRES, for approximating $f(A)b$. We focus our comparison on sFOM~\cite[Algorithm 1]{Guettel2022}, because sGMRES works best for Stieltjes functions and requires numerical quadrature otherwise.

While building a (nonorthonormal) basis $V_m$ (Algorithm~\ref{alg:NT} is used in the numerical examples of~\cite{Guettel2022}), sFOM also computes $\Theta V_m$ and $\Theta A V_m$ for a sketch matrix $\Theta$. After computing a QR factorization $\Theta V_m = S_m R_m$ of the sketched basis, sFOM returns the approximation
\begin{equation}\label{eq:sFOM}
    \hat f_m := V_m  \left ( R_m^{-1} f\left ( S_m^T \Theta A V_m R_m^{-1} \right ) S_m^T \Theta b\right ).
\end{equation}
If Algorithm~\ref{alg:fAb-cheap} was combined with \rev{the} whitening \rev{of} the basis obtained from Algorithm~\ref{alg:NT} (see Section~\ref{sec:NT}), the returned approximation would be 
mathematically equivalent to~\eqref{eq:sFOM}.
To see this, note that the whitened basis of Algorithm~\ref{alg:NT} takes the form $V_{m+1} R_{m+1}^{-1}$ and therefore the Arnoldi-like relation~\eqref{eq:arnoldilike} implies
\begin{equation*}
(\Theta V_m)^T \Theta A V_m = (\Theta V_m)^T\Theta V_m H_m + h_{m,m+1}(\Theta V_m)^T\Theta v_{m+1}e_{m+1}^T.
\end{equation*}
In turn, as the columns of $\Theta V_{m+1}$ are orthonormal, this simplifies to $S_m^T \Theta A V_m = H_m$. However, the advantage of sFOM over Algorithm~\ref{alg:fAb-cheap} combined with whitening is that in~\eqref{eq:sFOM} it is not needed to explicitly compute $V_m R_m^{-1}$.  
For both, Algorithm~\ref{alg:fAb-cheap} and sFOM it is difficult to derive meaningful a priori convergence statements; see Theorem~\ref{thm:lastcol} and  Corollary 2.1 in~\cite{Guettel2022}, respectively. In practice, both algorithms can be quite fast in some cases but in other cases their observed convergence behavior is erratic; see Figures~\ref{fig:gnutella}, \ref{fig:spectralProj}, and~\ref{fig:wikivote} below.

\begin{remark}
Another alternative view of sFOM is that it is mathematically equivalent to use Algorithm~\ref{alg:NT} for generating the basis, followed by obtaining the last column of $V_m^{\dagger} A V_m$ through solving the least-squares problem $\arg\min_{z \in \R^m} \|V_m z - v_{m+1}\|_2$ approximately with a sketch-and-solve approach, that is, $\arg\min_{z \in \R^m} \|\Theta(V_m z - v_{m+1})\|_2$.
\end{remark}

Table~\ref{tab:cost} compares the computational complexity of all algorithms discussed in this section. This reiterates the point that the $\mathcal O(nm^2)$ complexity can be avoided when using Algorithm~\ref{alg:NT}. 

\begin{table}[htbp]
\caption{\textbf{Cost comparison.} Arithmetic cost of approximately computing $f(A)b$
using an $m$-dimensional Krylov space, using Algorithms~\ref{alg:fAb-LS-BG}, \ref{alg:fAb-LS-NT}, \ref{alg:fAb-cheap}, Arnoldi, or sFOM. We assume that the truncation parameter $k$ in Algorithm~\ref{alg:NT} satisfies $k \ll m \ll n$ and we denote by $T_{\mathrm{matvec}}$ the cost of one matrix-vector multiplication with $A$.} \label{tab:cost}
\begin{center} \renewcommand{\arraystretch}{1.2}
\begin{tabular}{|l||c|c|c|c|c|c}
\hline
			& Matrix access				& Form $V_m$	& Sketch		& $f(H_m)$	& Form approx.	\\
\hline
Arnoldi	& $m T_{\mathrm{matvec}}$	& $nm^2$		& ---			& $m^3$		& $nm$			\\
Algorithm~\ref{alg:fAb-LS-BG}	& $m T_{\mathrm{matvec}}$	& $nm^2$		& ---			& $m^3$		& $nm$			\\
Algorithm~\ref{alg:fAb-LS-NT}	& $m T_{\mathrm{matvec}}$	& $nmk$		& $nm\log m$			& $m^3$		& $nm$	
\\
Algorithm~\ref{alg:fAb-cheap}& $m T_{\mathrm{matvec}}$	& $nmk$			& ---	&  $m^3$ 	& $nm$			\\
sFOM & $m T_{\mathrm{matvec}}$	&   $nmk$			& $nm \log m$	& $m^3$		&   $nm$			\\
\hline
\end{tabular}
\end{center}
\end{table}

\section{Convergence analysis}\label{sec:convergence}

In this section, we analyze the convergence of Algorithms~\ref{alg:fAb-LS-BG}, \ref{alg:fAb-LS-NT}, and~\ref{alg:fAb-cheap}. For this purpose,
let $\Pi_m$ denote the vector space of polynomials of  degree at most $m$ and  let  $W(A) = \{ z^* A z \mid z \in \C^n, \|z\|_2 = 1\}$ denote the numerical range of $A$. Given a set $\mathbb{D} \subseteq \C$, we define  $\|f\|_{\mathbb{D}}:=\sup\{|f(z)|: z\in \mathbb D\}$.

\subsection{Analysis of Algorithms~\ref{alg:fAb-LS-BG} and~\ref{alg:fAb-LS-NT}}

We have the following well known result that links the quality of the approximation~\eqref{eq:krylovproj} with polynomial approximation of $f$ on the numerical range of $A$.

\begin{theorem}\label{thm:convclassic}
%	\rrr{Consider a matrix $A \in \R^{n\times n}$, a vector $b \in \R^n$}
Let $A \in \R^{n\times n}$, $b \in \R^n$, and a function $f:\mathbb{D}\rightarrow \mathbb{C}$ be holomorphic on a domain $\mathbb{D}\subseteq \mathbb{C}$ containing $W(A)$. Then the approximation $f_m$ from~\eqref{eq:krylovproj} satisfies
	\begin{equation*}
		\|f(A)b - f_m \|_2 \le 2(1+\sqrt{2}) \min_{p \in \Pi_{m-1}} \|f-p\|_{W(A)}.
	\end{equation*}
\end{theorem}
\begin{proof}
    This is a classical result for the approximation given by the Arnoldi decomposition; see, e.g.,~\cite[Corollary 3.4]{Guettel2013}. Thanks to Lemma~\ref{lemma:basis}, the result extends to any approximation of the form~\eqref{eq:krylovproj}.
\end{proof}

Theorem~\ref{thm:convclassic} assumes that the least-squares problem~\eqref{eq:leastsquares} is solved exactly. 
In line 2 of Algorithms~\ref{alg:fAb-LS-BG} and~\ref{alg:fAb-LS-NT}, only an approximation by an iterative solver is used. The following lemma quantifies the impact of this approximation. 
\begin{lemma}\label{lemma:derivative}
	Let $\tilde y$ be an approximation to $y = \arg\min_{z\in\R^m} \|V_m z - v_{m+1}\|_2$. Assume that $f$ is holomorphic on a convex open set $\mathbb{D}$ containing $W(V_m^{\dagger} A V_m) \cup W(H_m + h_{m+1,m}\tilde{y} e_m^T)$.  Then 
	\begin{equation*}
		\|f(V_m^{\dagger} A V_m) - f(H_m + h_{m+1,m}\tilde{y} e_m^T)\|_F \le (1 + \sqrt{2})^2 \|f'\|_{\mathbb{D}}\|h_{m+1,m}(y - \tilde{y})\|_2,
	\end{equation*}
	where $\|\cdot\|_F$ denotes the Frobenius norm.
\end{lemma}
\begin{proof}
	Corollary 4.2 in~\cite{Beckermann2021} states that for any square matrices $C$ and $D$ of the same size, $\|f(C) - f(D)\|_F \le (1+\sqrt{2})^2 \|f'\|_{\mathbb{D}} \|C-D\|_F$ for a convex open set $\mathbb D$ that contains $W(C)\cup W(D)$. The lemma follows by applying this result to the matrices $V_m^{\dagger} A V_m$ and $H_m + h_{m+1,m}\tilde y r_m^T$. The difference of these two matrices has  norm  $\|h_{m+1,m}(y-\tilde y)e_m^T\|_F = \| h_{m+1,m} (y - \tilde y)\|_2$.
\end{proof}

As the norm of $V_m$ can be expected to remain modest in our algorithms, Lemma~\ref{lemma:derivative} implies that the error when solving the  least-squares problem will not get amplified significantly, provided that the derivative $f^\prime$ behaves nicely \rev{ and that the numerical ranges of the matrices $V_m^{\dagger} A V_m$ and $H_m + h_{m+1,m} \tilde y e_m^T$ do not become too large}. 

\subsection{Analysis of Algorithm~\ref{alg:fAb-cheap}}
\begin{theorem}\label{thm:lastcol}
	Considering the setting of Algorithm~\ref{alg:fAb-cheap}, assume that $f$ is holomorphic on a domain containing $W(V_m^{\dagger} A V_m)$ and $W(H_m)$. Then
	\begin{equation}\label{eq:bound5}
	    \|f(V_m^{\dagger} A V_m)e_1 - f(H_m)e_1\|_2 \le (2 + 2\sqrt{2}) \min_{p \in \Pi_{m-1}} \|f-p\|_{W(V_m^{\dagger} A V_m) \cup W(H_m)}.
	\end{equation}
\end{theorem}

\begin{proof}
	The matrices $V_m^{\dagger}A V_m$ and $H_m$ are upper Hessenberg and differ  only in their last columns. Therefore, for every polynomial $p \in \Pi_{m-1}$ it holds that $p(V_m^{\dagger}A V_m)e_1 = p(H_m)e_1$. This allows us to write
	\begin{align*}
		\| f(V_m^{\dagger}A V_m)e_1 - f(H_m)e_1 \|_2 & = \|f(V_m^{\dagger}A V_m) e_1 - p(V_m^{\dagger}A V_m)e_1 + p(H_m)e_1 - f(H_m)e_1 \|_2 \\
		& \le \|(f-p)(V_m^{\dagger}A V_m)e_1\|_2 + \|(f-p)(H_m)e_1\|_2 \\
		& \le \|(f-p)(V_m^{\dagger}A V_m)\|_2 + \|(f-p)(H_m)\|_2\\
		& \le (1 + \sqrt{2}) \left (\|f-p\|_{W(V_m^{\dagger}A V_m)} + \|f-p\|_{W(H_m)}\right ) ,
	\end{align*}
	where the third inequality is due to a result by Crouzeix and Palencia~\cite{Crouzeix2017}.
\end{proof}

An immediate consequence of Theorem~\ref{thm:lastcol} is that the output of Algorithm~\ref{alg:fAb-cheap} is exact for all polynomials of degree up to $m-1$ (and such an output coincides with~\eqref{eq:krylovproj}). 
Such exactness property holds for all algorithms considered in this paper, namely Algorithm~\ref{alg:fAb-LS-BG}, \ref{alg:fAb-LS-NT}, and~\ref{alg:fAb-cheap} (and hence also sFOM).

By the triangle inequality, we have that
\begin{multline*}
    \left \lVert f(A)b - \|b\|_2 V_m f(H_m) e_1 \right \rVert_2 \\
    \le \|f(A)b - \|b\|_2 V_m f(V_m^{\dagger} A V_m)e_1\|_2 + \|b\|_2 \|V_m f(V_m^{\dagger} A V_m)e_1 - V_m f(H_m)e_1\|_2.
\end{multline*}
We can now use Theorem \ref{thm:convclassic} to bound the first term and Theorem~\ref{thm:lastcol} for the second term, in addition to the fact that the spectral norm of $V_m$ is bounded by $\sqrt{m}$ because $V_m$ has columns of unit norm.

While Theorem~\ref{thm:convclassic} ensures that the output of Algorithms~\ref{alg:fAb-LS-BG} and~\ref{alg:fAb-LS-NT} are close to $f(A)b$ whenever $f$ is well approximated by a polynomial of small degree on $W(A)$, in order to obtain a good a priori bound on the error of Algorithm~\ref{alg:fAb-cheap} we need a good polynomial approximation on a larger set, that is, the union of $W(V_m^{\dagger}A V_m)$ and $W(H_m)$. In general, the right-hand-side of~\eqref{eq:bound5} is difficult to control and, in practice, we sometimes observe that Algorithm~\ref{alg:fAb-cheap} fails to converge or the convergence is slowed down (see, e.g., Figures~\ref{fig:ConvDiff2} and~\ref{fig:wikivote}).

\subsection{Discussion on numerical stability}\label{sec:stability}

Rigorous numerical stability analysis of any of the methods discussed in this work appears to be challenging. The \rev{evaluation} of the function of a small matrix $f(H_m)$ is already a nontrivial computation, for which (backward or forward) stability is established only for limited classes of matrices and functions $f$ (e.g., normal matrices with $f=\exp$~\cite[Chapter~10]{Higham2008}). 
To nevertheless gain some insight into the numerical behavior of Algorithms~\ref{alg:fAb-LS-BG} and~\ref{alg:fAb-LS-NT}, particularly in terms of how the conditioning of $V_m$ affects the stability, let us \emph{assume} that $f(H_m)$ is computed in a backward stable fashion, that is, its computed approximation can be written as $f(H_m+\Delta H)$, where $\|\Delta H\|_2=\mathcal O(\epsilon\|H_m\|_2)$ with the machine precision $\epsilon$. Then in the computation of $f(V_m^\dagger A V_m)$ in~\eqref{eq:krylovproj}, 
the backward error in computing $f(V_m^\dagger A V_m)$ becomes
$\epsilon\|V_m^\dagger A V_m\|_2$. Letting $V_m=Q_mR_m$ be a QR factorization, we have 
$\epsilon\|V_m^\dagger A V_m\|_2 = \|R_m^{-1}Q_m^T A Q_mR_m\|_2
\leq \|R_m^{-1}\|_2\|Q_m^T A Q_m\|_2\|R_m\|_2=\kappa_2(R_m)\|Q_m^T A Q_m\|_2$. This suggests that an ill-conditioned basis $V_m$ may \rev{negatively} impact the numerical stability of the algorithms (in practice, the approximation is often good unless $V_m$ is numerically rank-deficient). 
It also suggests that a well-conditioned $V_m$ would yield an approximation similar to the one obtained via full Arnoldi orthogonalization (wherein $V_m=Q_m$), which we observe in our experiments. 
\rev{In practice, convergence is sometimes observed even when $V_m$ is highly ill-conditioned; this has been studied extensively for the Lanczos method~\cite{musco2018stability}, and steepest descent (restarted Arnoldi)~\cite{afanasjew2008generalization} when $A$ is symmetric. However, a complete analysis for the nonsymmetric and randomized case is an open problem.}

\section{Numerical experiments}\label{sec:numexp}

We now test the proposed Algorithms~\ref{alg:fAb-LS-BG}, \ref{alg:fAb-LS-NT}, and~\ref{alg:fAb-cheap} on a variety of examples; these are indicated as ``BG-f(A)'', ``NT1-f(A)'', and ``NT2-f(A)'' in the legends, respectively \rev{(the letters BG and NT are the initials of the authors of the papers~\cite{Balabanov2020,Nakatsukasa2021} from where Algorithms~\ref{alg:BG} and~\ref{alg:NT} are taken)}. The baseline method is ``Arnoldi'', which corresponds to~\eqref{eq:standardfm}.
For some examples, we also compare with sFOM (\cite[Algorithm 1]{Guettel2022}) where the basis is obtained by Algorithm~\ref{alg:NT}. In all our examples we set the truncation parameter $k = 2$ in Algorithm~\ref{alg:NT}. The sketch matrix $\Theta$ is an SRHT of size $2m \times n$, where $m$ is the dimension of the Krylov subspace we are considering, unless we indicate otherwise.

All algorithms were implemented and executed in MATLAB R2022a. For the evaluation of $f$ on the compressed matrices we used Matlab \texttt{expm} and \texttt{sqrtm} functions.
We report the runtime and the relative accuracy $\frac{\|f(A)b-\widehat{f(A)b}\|_2}{\|f(A)b\|_2}$, where $\widehat{f(A)b}$ is the computed approximation to $f(A)b$ using one of the algorithms. 
The ``exact'' value of $f(A)b$ is computed using \texttt{expm} or \texttt{sqrtm} when the size of the matrix is smaller than $10^4$, otherwise we approximate the exact value running sufficiently many iterations of the Arnoldi algorithm.
For the implementation of Algorithm~\ref{alg:BG} we have used the implementation available at \url{https://github.com/obalabanov/randKrylov}, which uses a subsampled Walsh-Hadamard transform for sketching. The least-squares problems were solved using Matlab's \texttt{lsqr} function with tolerance parameter $10^{-6}$ and for Blendenpik we used the code at \url{https://github.com/haimav/Blendenpik}. The code for sFOM was taken from \url{https://github.com/MarcelSchweitzer/sketched_fAb}.
 All numerical experiments in this work have been run on a laptop with \rev{an} eight-core Intel Core i7-8650U 1.90 GHz CPU and 16 GB of RAM. The code for reproducing the experiments is available at \url{https://github.com/Alice94/RandMatrixFunctions}.

\subsection{Behavior of Algorithms~\ref{alg:fAb-LS-BG}, \ref{alg:fAb-LS-NT}, and~\ref{alg:fAb-cheap}}

\subsubsection{Convection-diffusion problem}\label{sec:convdiff1}
As a first example, we consider the matrix $A$ coming from the discretization of the convection-diffusion problem considered in~\cite[Section 3.1]{Botchev2020}:
\begin{align*}
	& L[u] = -(D_1 u_x)_x - (D_2 u_y)_y + Pe \left ( \frac{1}{2} (v_1 u_x + v_2 u_y) + \frac{1}{2} ((v_1 u)_x + (v_2 u)_y)\right ),\\
	&D_1(x,y) = \begin{cases} 
		1000 & (x,y) \in [0.25, 0.75]^2\\
		1 & \text{otherwise}
		\end{cases} \qquad D_2(x,y) = \frac{1}{2} D_1(x,y)\\
	&v_1(x,y) = x+y, \qquad v_2(x,y) = x-y.
\end{align*}
 We take homogeneous Dirichlet boundary conditions and Peclet number $Pe = 200$, and we consider a discretization by finite differences on a mesh of size $352 \times 352$. The function $f$ is the exponential and the starting vector is set to the values of the function $\sin(\pi x) \sin(\pi y)$ on the finite-difference mesh and then normalized as $b := b/\|b\|_2$. The results are reported in Figure~\ref{fig:ConvDiff1}. We observe that all our proposed methods are stable in this example and they produce approximately the same result. As expected, Algorithm~\ref{alg:fAb-cheap} is by far the fastest one. 

\begin{figure}[ht]
	\centering
	\includegraphics[width=0.9\textwidth]{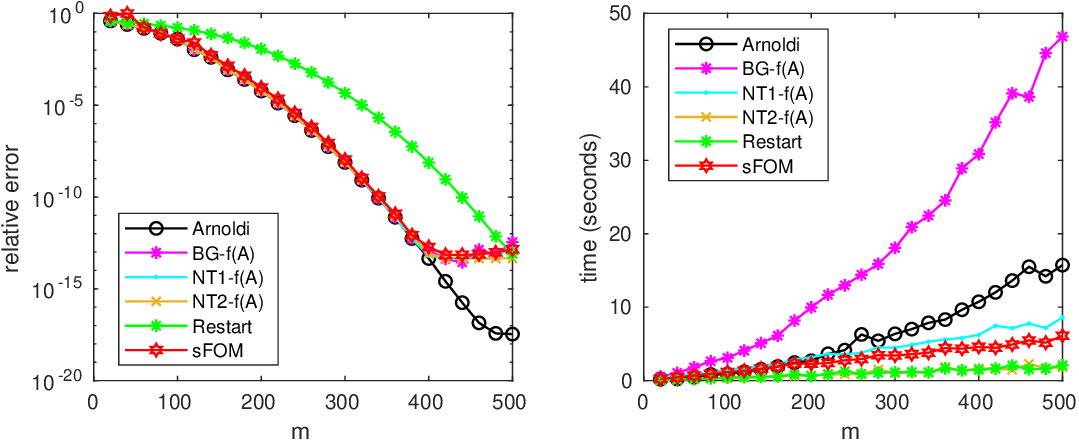}
	\caption{Exponential of the matrix from the convection-diffusion problem in Section~\ref{sec:convdiff1}.}
	\label{fig:ConvDiff1}
\end{figure}

\subsubsection{Option pricing problem}\label{sec:option}
Another example is illustrated in Figure~\ref{fig:toeplitz}. We compute the exponential of the nonsymmetric Toeplitz matrix from Section 6.3 in~\cite{Kressner2018}, size $n = 7000$ and we take $b$ to be a normalized random vector. 

\begin{figure}
	\centering
	\includegraphics[width=0.9\textwidth]{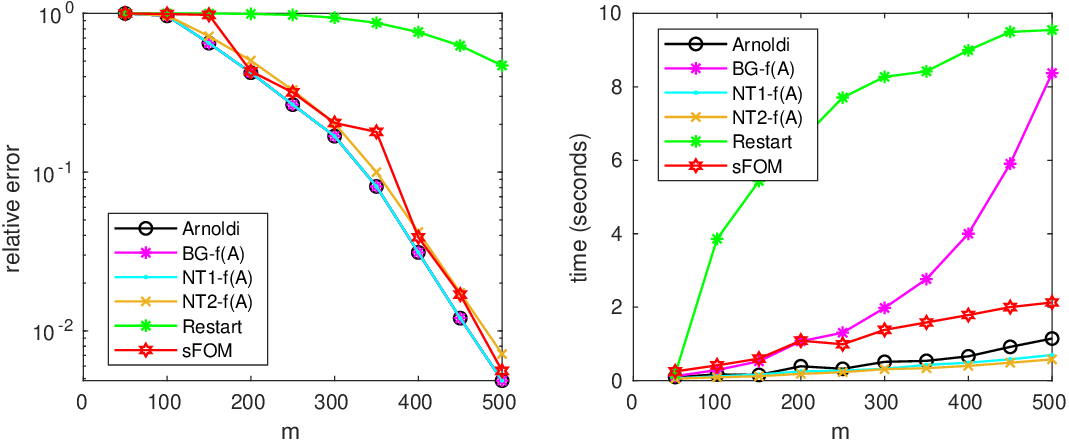}
	\caption{Computing the exponential of the nonsymmetric Toeplitz matrix from \rev{Section~\ref{sec:option}}.}
	\label{fig:toeplitz}
\end{figure}

\subsubsection{Whitening the basis}\label{sec:whitening}

We now consider two examples in which Algorithm~\ref{alg:NT} alone does not give a well-conditioned basis, and therefore whitening (see Section~\ref{sec:NT}) is needed to make the algorithm work. To check the condition number of the sketch $\Theta V_i$ we use Matlab's \texttt{condest} function and we perform whitening if this exceeds $1000$.

Let us compute the square root of the graph Laplacian of the matrix \texttt{p2p\_Gnutella08} from~\url{https://sparse.tamu.edu}, which was considered in~\cite{Benzi2020}. The matrix has size $n = 6301$, is nonsymmetric, and all its eigenvalues have positive real part. We choose  $b$ to be a random unit vector. The results are reported in Figure~\ref{fig:gnutella}.
 Around the 60th iteration of Algorithm~\ref{alg:NT}, the quality of the basis (in terms of conditioning) starts to deteriorate \rev{quickly}, and indeed the solid cyan line does not converge. However, when whitening is performed, the convergence \rev{is restored}. 
\begin{figure}
	\centering
	\includegraphics[width=0.7\textwidth]{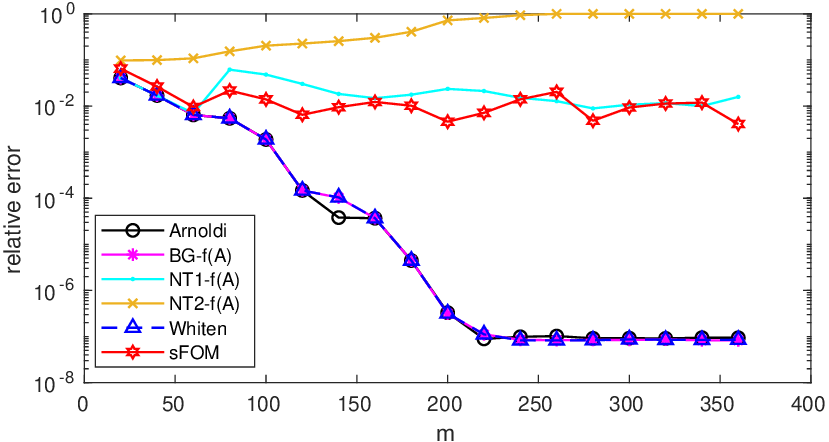}
	\caption{Square root of the graph Laplacian of the matrix from Section~\ref{sec:whitening}.}
	\label{fig:gnutella}
\end{figure}

As another example, we consider the matrix \texttt{bfw782a} from \url{https://math.nist.gov/MatrixMarket/data/NEP/bfwave/bfw782a.html}, a random unit vector $b$ and the sign function. To make computations more stable, to apply the sign function to a vector we use the formula $f(A)b = (A^2)^{-1/2} (Ab)$~\cite[Chap.~5]{Higham2008}. In this case, the basis obtained from Algorithm~\ref{alg:NT} is badly conditioned and whitening is needed. The results are reported in Figure~\ref{fig:spectralProj}.

\begin{figure}
	\centering
	\includegraphics[width=0.7\textwidth]{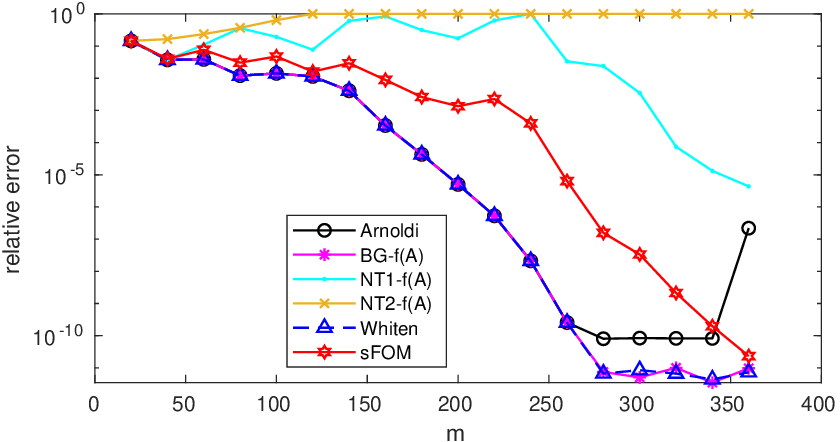}
	\caption{Sign of the matrix in Section~\ref{sec:whitening}.}
	\label{fig:spectralProj}
\end{figure}

\subsubsection{Timing of Algorithm~\ref{alg:fAb-LS-BG} vs.\ Arnoldi}\label{sec:timingBG}

We observed that Algorithm~\ref{alg:fAb-LS-BG} is slower than Arnoldi for the example in Figure~\ref{fig:ConvDiff1}. As discussed in~\cite{Balabanov2020}, the benefits of Algorithm~\ref{alg:BG} compared to Arnoldi are only visible for moderately large values of $m$, the dimension of the Krylov subspace. In Algorithm~\ref{alg:fAb-LS-BG}, the construction of the basis needs to be combined with an additional step -- the solution of the least-squares problem. In Figure~\ref{fig:invsqrt3DLap} we compute the inverse square root of a 3D Laplacian matrix of the form
\begin{equation*}
	A = T_N \otimes I_N \otimes I_N + I_N \otimes T_N \otimes I_N + I_N \otimes I_N \otimes T_N,
\end{equation*}
where $T_N = \mathrm{tridiag}(-1, 2, 1)$ and $I_N$ is the identity matrix with size $N = 80$ \rev{(therefore $n=512000$)}. The vector $b$ is a randomly chosen unit vector. The sketch matrix $\Theta$ is of size $\lfloor 1.05 m\rfloor \times n$ in this example. We observe that, for large enough values of $m$, Algorithm~\ref{alg:fAb-LS-BG} is indeed faster than Arnoldi.

\begin{figure}
	\centering
	\includegraphics[width=0.9\textwidth]{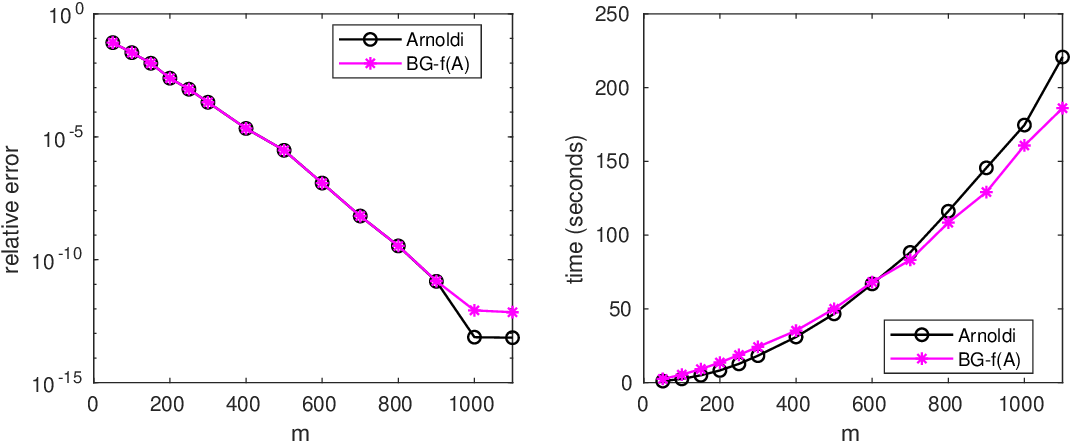}
	\caption{Computing the inverse square root of the 3D Laplacian from Section~\ref{sec:timingBG}.}
	\label{fig:invsqrt3DLap}
\end{figure}

\subsection{Comparison with~\cite{Guettel2022}}\label{sec:guettel}

We also include a comparison of our algorithms with sFOM~\cite{Guettel2022} on the examples from~\cite[Sections 5.1 and 5.2]{Guettel2022}. In Figure~\ref{fig:ConvDiff2} we consider the inverse square root of a matrix coming from a convection-diffusion problem (different from the one in Figure~\ref{fig:ConvDiff1}) and in Figure~\ref{fig:wikivote} we consider the exponential of the wiki-vote matrix, of size $n = 8297$, from the SNAP collection~\url{https://sparse.tamu.edu/SNAP/wiki-Vote}. Observe that, in Figure~\ref{fig:wikivote}, Algorithm~\ref{alg:fAb-cheap} stagnates until $m=25$, even though the basis obtained from Algorithm~\ref{alg:NT} is numerically full rank, perhaps reflecting the discussion on stability in Section~\ref{sec:stability}.

\begin{figure}
	\centering
	\includegraphics[width=0.9\textwidth]{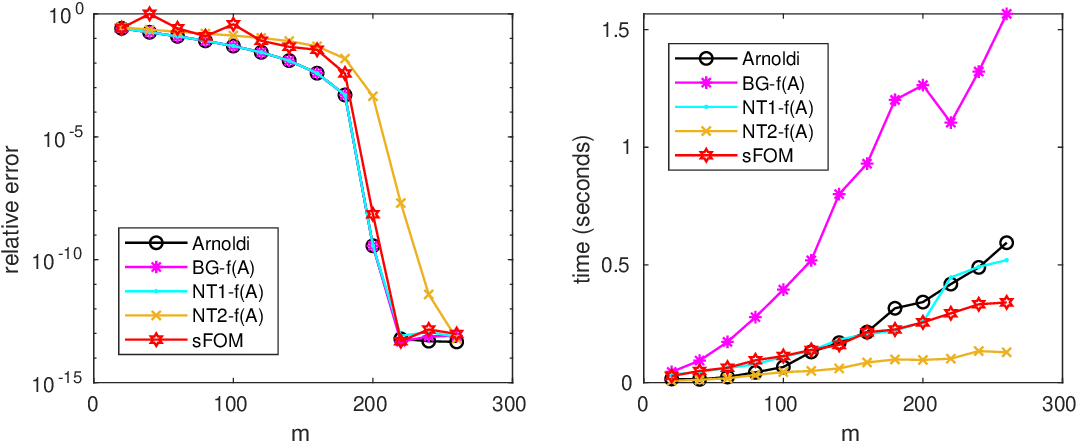}
	\caption{Convection diffusion example from Section~\ref{sec:guettel}, size $n=10^4$, $f$ is the inverse square root.
	}
	\label{fig:ConvDiff2}
\end{figure}

\begin{figure}
	\centering
	\includegraphics[width=0.7\textwidth]{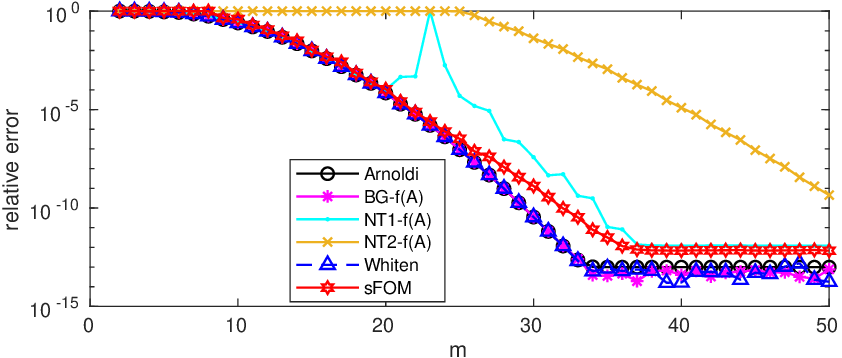}
	\caption{Exponential of the wiki-vote matrix considered in Section~\ref{sec:guettel} and in~\cite[Section 5.2]{Guettel2022}.}
	\label{fig:wikivote}
\end{figure}

\section{Conclusion}\label{sec:conclusion}

In this work we proposed two ways of speeding up Krylov subspace methods for computing $f(A)b$ for a nonsymmetric matrix $A$ and general function $f$.
For well-behaved problems, where the condition number of the constructed Krylov subspace basis does not deteriorate, it is best to run Algorithm~\ref{alg:fAb-cheap} until convergence. Once 
this condition number gets very large, our experiments indicate that it is beneficial to perform whitening and then switch to the more expensive Algorithm~\ref{alg:fAb-LS-BG}. The algorithm sFOM proposed in~\cite{Guettel2022} offers a compromise between the different methods, in being more stable (but slower) than Algorithm~\ref{alg:fAb-cheap} and faster (but less stable) than Algorithm~\ref{alg:fAb-LS-BG}.

\paragraph{Acknowledgments.} The authors thank Laura Grigori, Stefan G\"uttel, and Marcel Schweitzer for helpful discussions on this work. AC thanks the Mathematical Institute at the University of Oxford for hosting her during a research visit that initiated this work.

	\bibliographystyle{abbrv}
	\bibliography{Bib}

\begin{thebibliography}{10}

\bibitem{afanasjew2008generalization}
{\sc M.~Afanasjew, M.~Eiermann, O.~G. Ernst, and S.~G{\"u}ttel}, {\em A
  generalization of the steepest descent method for matrix functions},
  Electron. Trans. Numer. Anal, 28 (2008), pp.~206--222.

\bibitem{Blendenpik}
{\sc H.~Avron, P.~Maymounkov, and S.~Toledo}, {\em Blendenpik: Supercharging
  {LAPACK}'s least-squares solver}, SIAM J. Sci. Comput., 32 (2010),
  pp.~1217--1236, \url{https://doi.org/10.1137/090767911}.

\bibitem{Balabanov2020}
{\sc O.~Balabanov and L.~Grigori}, {\em Randomized {G}ram-{S}chmidt process
  with application to {GMRES}}, SIAM J. Sci. Comput., 44 (2022),
  pp.~A1450--A1474, \url{https://doi.org/10.1137/20M138870X}.

\bibitem{Beckermann2021}
{\sc B.~Beckermann, A.~Cortinovis, D.~Kressner, and M.~Schweitzer}, {\em
  Low-rank updates of matrix functions {II}: Rational {K}rylov methods}, SIAM
  J. Numer. Anal., 59 (2021), pp.~1325--1347,
  \url{https://doi.org/10.1137/20M1362553}.

\bibitem{Benzi2020}
{\sc M.~Benzi and P.~Boito}, {\em Matrix functions in network analysis},
  GAMM-Mitt., 43 (2020), pp.~e202000012, 36.

\bibitem{Benzi2013}
{\sc M.~Benzi, P.~Boito, and N.~Razouk}, {\em Decay properties of spectral
  projectors with applications to electronic structure}, SIAM Rev., 55 (2013),
  pp.~3--64, \url{https://doi.org/10.1137/100814019}.

\bibitem{Botchev2020}
{\sc M.~A. Botchev and L.~A. Knizhnerman}, {\em A{RT}: {A}daptive residual-time
  restarting for {K}rylov subspace matrix exponential evaluations}, J. Comput.
  Appl. Math., 364 (2020), pp.~112311, 14,
  \url{https://doi.org/10.1016/j.cam.2019.06.027}.

\bibitem{Crouzeix2017}
{\sc M.~Crouzeix and C.~Palencia}, {\em The numerical range is a
  {$(1+\sqrt{2})$}-spectral set}, SIAM J. Matrix Anal. Appl., 38 (2017),
  pp.~649--655, \url{https://doi.org/10.1137/17M1116672}.

\bibitem{Davies2003}
{\sc P.~I. Davies and N.~J. Higham}, {\em A {S}chur-{P}arlett algorithm for
  computing matrix functions}, SIAM J. Matrix Anal. Appl., 25 (2003),
  pp.~464--485, \url{https://doi.org/10.1137/S0895479802410815}.

\bibitem{Eiermann2006}
{\sc M.~Eiermann and O.~G. Ernst}, {\em A restarted {K}rylov subspace method
  for the evaluation of matrix functions}, SIAM J. Numer. Anal., 44 (2006),
  pp.~2481--2504, \url{https://doi.org/10.1137/050633846}.

\bibitem{Estrada2010}
{\sc E.~Estrada and D.~J. Higham}, {\em Network properties revealed through
  matrix functions}, SIAM Rev., 52 (2010), pp.~696--714,
  \url{https://doi.org/10.1137/090761070}.

\bibitem{Frommer2014}
{\sc A.~Frommer, S.~G\"{u}ttel, and M.~Schweitzer}, {\em Efficient and stable
  {A}rnoldi restarts for matrix functions based on quadrature}, SIAM J. Matrix
  Anal. Appl., 35 (2014), pp.~661--683,
  \url{https://doi.org/10.1137/13093491X}.

\bibitem{Guettel2013}
{\sc S.~G\"{u}ttel}, {\em Rational {K}rylov approximation of matrix functions:
  numerical methods and optimal pole selection}, GAMM-Mitt., 36 (2013),
  pp.~8--31, \url{https://doi.org/10.1002/gamm.201310002}.

\bibitem{Guettel2020}
{\sc S.~G\"{u}ttel, D.~Kressner, and K.~Lund}, {\em Limited-memory polynomial
  methods for large-scale matrix functions}, GAMM-Mitt., 43 (2020),
  pp.~e202000019, 19.

\bibitem{Guettel2022}
{\sc S.~G{\"u}ttel and M.~Schweitzer}, {\em Randomized sketching for {K}rylov
  approximations of large-scale matrix functions}, arXiv preprint
  arXiv:2208.11447,  (2022).

\bibitem{Higham2008}
{\sc N.~J. Higham}, {\em Functions of matrices}, SIAM, PA, 2008,
  \url{https://doi.org/10.1137/1.9780898717778}.

\bibitem{Kressner2018}
{\sc D.~Kressner and R.~Luce}, {\em Fast computation of the matrix exponential
  for a {T}oeplitz matrix}, SIAM J. Matrix Anal. Appl., 39 (2018), pp.~23--47,
  \url{https://doi.org/10.1137/16M1083633}.

\bibitem{Martinsson2020}
{\sc P.-G. Martinsson and J.~A. Tropp}, {\em Randomized numerical linear
  algebra: {F}oundations and algorithms}, Acta Numerica, 29 (2020),
  pp.~403--572.

\bibitem{musco2018stability}
{\sc C.~Musco, C.~Musco, and A.~Sidford}, {\em Stability of the {L}anczos
  method for matrix function approximation}, in Proceedings of the Twenty-Ninth
  Annual ACM-SIAM Symposium on Discrete Algorithms, SIAM, 2018, pp.~1605--1624.

\bibitem{Nakatsukasa2021}
{\sc Y.~Nakatsukasa and J.~A. Tropp}, {\em Fast \& accurate randomized
  algorithms for linear systems and eigenvalue problems}, arXiv preprint
  arXiv:2111.00113,  (2021).

\bibitem{Paige1982}
{\sc C.~C. Paige and M.~A. Saunders}, {\em L{SQR}: an algorithm for sparse
  linear equations and sparse least squares}, ACM Trans. Math. Software, 8
  (1982), pp.~43--71, \url{https://doi.org/10.1145/355984.355989}.

\bibitem{Saad1992}
{\sc Y.~Saad}, {\em Numerical methods for large eigenvalue problems},
  Algorithms and Architectures for Advanced Scientific Computing, Manchester
  University Press, Manchester; Halsted Press [John Wiley \& Sons, Inc.], New
  York, 1992.

\bibitem{Saad2003}
{\sc Y.~Saad}, {\em Iterative methods for sparse linear systems}, SIAM, PA,
  second~ed., 2003, \url{https://doi.org/10.1137/1.9780898718003}.

\bibitem{Tropp2011}
{\sc J.~A. Tropp}, {\em Improved analysis of the subsampled randomized
  {H}adamard transform}, Adv. Adapt. Data Anal., 3 (2011), pp.~115--126,
  \url{https://doi.org/10.1142/S1793536911000787}.

\bibitem{Ciegis2018}
{\sc R.~\v{C}iegis, V.~Starikovi\v{c}ius, S.~Margenov, and R.~Kriauzien\.{e}},
  {\em A comparison of accuracy and efficiency of parallel solvers for
  fractional power diffusion problems}, in Parallel processing and applied
  mathematics. {P}art {I}, vol.~10777 of Lecture Notes in Comput. Sci.,
  Springer, Cham, 2018, pp.~79--89,
  \url{https://doi.org/10.1007/978-3-319-78024-5}.

\end{thebibliography}

\end{document}